\documentclass[12pt,reqno]{article}

\usepackage{amssymb}
\usepackage{amsmath}
\usepackage{amsthm}
\usepackage{amssymb}
\usepackage{euscript}
\usepackage{footnote}

\newtheorem{lemma}{Lemma}
\newtheorem{theorem}{Theorem}
\newtheorem{corollary}{Corollary}

\theoremstyle{remark}
\newtheorem{remark}{Remark}

\newcommand{\dist}{\operatorname{dist}}
\newcommand{\diam}{\operatorname{diam}}

\newcommand{\olim}{\operatornamewithlimits{\overline{lim}}}

\begin{document}

\title{\textbf{A Resonance Theorem\\ for a Family of Translation Invariant Differentiation Bases}}

\author{\textbf{Giorgi G. Oniani}}

\date{\fontsize{10}{12pt}\selectfont
    \begin{center}
    \textit{Department of Mathematics, Akaki Tsereteli State University, \\ 59 Tamar Mepe St., Kutaisi 4600, Georgia} \\
    \textit{E-mail: \texttt{oniani@atsu.edu.ge}}
\end{center}}
\maketitle



\begin{abstract}
A resonance theorem providing existence of functions that are counterexamples for all members of a given family of translation invariant differentiation bases is proved. Applications of the theorem to Zygmund problem on a choice of coordinate axes are given.
\end{abstract}

\bigskip
\bigskip

\section{Definitions and notation}
\label{sec:1}

A mapping $B$ defined on $\mathbb{R}^n$ is said to be a \emph{differen\-tiation basis} if for every $x\in \mathbb{R}^n$, $B(x)$ is a family  of bounded measurable sets with positive measure and containing $x$, such that there exists a sequence $R_k\in B(x)$ $(k\in\mathbb{N})$ with $\lim\limits_{k\rightarrow\infty}\diam R_k=0$.

For $f\in L(\mathbb{R}^n)$, the numbers
$$  \overline{D}_B\Big(\int f,x\Big)=\mathop{\overline{\lim}}\limits_{\substack{R\in B(x) \\ \diam R\to 0}}
                    \frac{1}{|R|} \int\limits_R f \;\; \text{and} \;\;
        \underline{D}{\,}_B\Big(\int f,x\Big)=
                \mathop{\underline{\lim}}\limits_{\substack{R\in B(x) \\ \diam R\to 0}} \frac{1}{|R|} \int\limits_R f $$
are called \emph{the upper and the lower derivative,} respectively, \emph{of the integral of $f$ at a point $x$}. If the upper and the lower derivative coincide, then their combined value is called the \emph{derivative of $\int f$ at a point $x$} and denoted by $D_B(\int f,x)$. We say that the\emph{ basis  $B$ differentiates} $\int f$ (or $\int f$ is differentiable with respect to $B$) if $\overline{D}_B(\int f,x)=\underline{D}{\,}_B(\int f,x)=f(x)$ for almost all $x\in \mathbb{R}^n$. If this is true for each $f$ in the class of functions $X$ we say that $B$ differentiates~$X$.

The \emph{maximal operator $M_B$} and  \emph{truncated maximal operator $M_B^{r}$ $(r>0)$ corresponding to a basis }$B$ are  defined as follows:
\begin{align*}
    M_B(f)(x) & =\sup_{R\in B(x)} \frac{1}{|R|} \int\limits_R |f|, \\
    M_B^r(f)(x) & =\sup_{\substack{R\in B(x) \\\diam R<r}}\frac{1}{|R|}\int\limits_R |f|,
\end{align*}
where $f\in L_{loc}(\mathbb{R}^n)$ and $x\in\mathbb{R}^n$.

By $\textbf{I}_n^{k}$ $(2\leq k\leq n)$ we will denote the basis such that $\textbf{I}_n^{k}(x)$ $(x\in \mathbb{R}^n)$ consists of all $n$-dimensional intervals  lengthes  of which edges take not more then $k$ different values and which contain $x$. The basis $\textbf{I}_n^{n}$ will be  denoted by $\textbf{I}$. The differentiation with respect to $\textbf{I}$ is called \emph{strong differentiation}.

A basis $B$ is called:
\begin{itemize}
\item \emph{translation invariant}(briefly, ${\rm TI}$-basis) if $B(x)=\{x+I: I\in B(0)\}$ for every $x\in\mathbb{R}^n$;

\item \emph{homothecy invariant}(briefly, ${\rm HI}$-basis) if for every $x\in\mathbb{R}^n, R\in B(x)$ and a homothecy $H$ with the centre at $x$ we have that $H(R)\in B(x)$;

\item \emph{formed of sets from the class} $\Delta$ if $\overline{B}\subset \Delta$.

\item \emph{convex} if it is formed of the class of all convex sets.
\end{itemize}

Denote by $\Gamma_n$ the family of all rotations in the space $\mathbb{R}^n$.

Let $B$ be a basis in $\mathbb{R}^n$ and $\gamma\in \Gamma_n$.  The $\gamma$-\emph{rotated basis} $B$ is defined as follows
$$  B(\gamma)(x)=\big\{x+\gamma(I - x):\;I\in B(x)\big\} \;\; (x\in\mathbb{R}^n).       $$

For  an increasing  function $\Phi: (0,\infty)\to (0,\infty)$ and a measurable set $E\subset\mathbb{R}^n$ by $\Phi(L)(E)$($[\Phi(L)](E)$) we denote the class of all measurable functions $f:\mathbb{R}^n\to \mathbb{R}$ such that $\{f\neq 0\}\subset E$ and $\int\limits_{\{f\neq 0\}}\Phi(|f|)<\infty$ ($\int\limits_{\{f\neq 0\}}\Phi(|f|/h)<\infty$ for some $h\geq 1$).

A function $\Phi: (0,\infty)\rightarrow (0,\infty)$ is said to satisfy $\Delta_2$-\emph{condition at infinity} if there are $c>0$ and $\tau> 0$ such that $\Phi(2t)\leq c\Phi(t)$ for every $t> \tau$.

We say that an increasing function $\Phi: (0,\infty)\rightarrow (0,\infty)$ is \emph{non-regular} if $\mathop{\overline{\lim}}\limits_{t\rightarrow\infty} \frac{\Phi(t)}{t}=\infty. $

The unit cube $(0,1)^n$ will be denoted by $\mathbb{G}^n$.

Let $\Phi: (0,\infty)\rightarrow (0,\infty)$ be an increasing function. It is easy to check that:
\begin{itemize}
\item[1)] If $\Phi$ satisfies $\Delta_2$-condition at infinity, then $[\Phi(L)](\mathbb{G}^n)=\Phi(L)(\mathbb{G}^n)$;

\item[2)] The class $L\setminus [\Phi(L)](\mathbb{G}^n)$ is non-empty if and only if $\Phi$ is non-regular.
\end{itemize}

The family of all diadic intervals of order $m\in\mathbb{Z}^n$ we will denote by $W_m$, i.e.,
$$  W_m=\bigg\{\underset{j=1}{\overset{n}{\times}} \Big(\frac{k_j}{2^{m_j}}\,,\frac{k_j+1}{2^{m_j}}\Big):\;k_1,\dots,k_n \in\mathbb{Z}^n\bigg\}.  $$
By $H_m$ ($m\in\mathbb{Z}^n$)  it will be denoted the family of all possible  unions of intervals from $W_m$.

 Below everywhere it will be assumed that the dimension $n$ is greater then~$1$.

\section{Main result}
\label{sec:2}

Saks \cite{1} and Busemann and Feller \cite{2} constructed a function $f\in L(\mathbb{R}^n)$ whose integral is not strongly differentiable.

Zygmund \cite[p.~99]{3} posed the problem: \emph{Is it possible for arbitrary function $f\in L(\mathbb{G}^2)$  to choose a rotation $\gamma\in\Gamma_2$ so that}  $\textbf{I}(\gamma)$ \emph{differentiates} $\int f$?

Marstrand \cite{4} gave a negative answer to the problem, namely, constructed a non-negative function $f\in L(\mathbb{G}^2)$ such that for every $\gamma\in\Gamma_2$,
$$  \overline{D}_{\textbf{I}(\gamma)}\Big(\int f,x\Big)=\infty \;\;\text{almost everywhere on $\mathbb{G}^2$}.      $$

Developing Marstrand approach below we will prove a resonance  theorem providing existence of functions that are counterexamples for all members of a given family of translation invariant bases. Applications of the theorem to Zygmund problem are given also.

The work is a revised version of \S~II.1 from the monograph \cite{5}.

Let $\Lambda$ be a non-empty family of translation invariant differentiation bases in $\mathbb{R}^n$ and let $\Phi: (0,\infty)\to (0,\infty)$ be a some function. We will say that $\Lambda$ has $M_\Phi$-\emph{property} if for every $h>1$ and $\varepsilon>0$ there exist a set $E\subset\mathbb{R}^n$ of positive measure, sets $P_B$ $(B\in\Lambda)$ and an interval $Q$ such that:
\begin{itemize}
\item[1)] $P_B\subset\{M_B^{(\varepsilon)}(h\chi_{{}_E})>1\}$ $(B\in\Lambda)$;

\item[2)] $\{P_B:\;B\in\Lambda\}\subset H_m$ for some $m\in\mathbb{N}^n$;

\item[3)] $|P_B|\geq c\Phi(h)|E|$ $(B\in\Lambda)$;

\item[4)] $E\subset Q$ and $P_B\subset Q$ $(B\in\Lambda)$;

\item[5)] $\diam Q<\varepsilon$;

\item[6)] $|E|\geq c(h)|Q|$,
\end{itemize}
where $c>0$, $c$ does not depend on $h$ and $\varepsilon$, $c(h)\in(0,1)$ and $c(h)$ does not depend on $\varepsilon$.

For a non-empty family of differentiation bases $\Lambda$ by $S_\Lambda$ denote the class of all functions $f\in L(\mathbb{G}^n)$ for which
$$  \overline{D}_B\Big(\int f,x\Big)=\infty \;\;\text{almost everywhere on $\mathbb{G}^n$}  $$
for every $B\in\Lambda$.

\begin{theorem}\label{th:1}
Let $\Lambda$ be a non-empty family of translation invariant differentiation bases in $\mathbb{R}^n$ and let $\Phi: (0,\infty)\rightarrow (0,\infty)$ be a non-regular increasing function.  If $\Lambda$ has $M_\Phi$-property, then for every  $f\in L\setminus [\Phi(L)](\mathbb{G}^n)$   there exists a measure preserving and invertible mapping $\omega:\mathbb{R}^n\to\mathbb{R}^n$ such that  $\{x:\omega(x)\neq x\}\subset \mathbb{G}^n$ and $|f|\circ\omega\in S_\Lambda$. In particular, if $\Phi$ additionally satisfies $\Delta_2$-condition at infinity, then the same conclusion is valid for every $f\in L\setminus \Phi(L)(\mathbb{G}^n)$.
\end{theorem}

\section{Auxiliary propositions}
\label{sec:3}

For $m\in\mathbb{N}^n$ denote
$$  W_m^{*}=\big\{E\in W_m:\;E\subset \mathbb{G}^n\big\}, \quad H_m^{*}=\big\{E\in H_m:\;E\subset \mathbb{G}^n\big\}.       $$

\begin{lemma}\label{lem:1}
Let $\Lambda$ be a non-empty family of translation invariant differentiation bases in $\mathbb{R}^n$ and $\Phi: (0,\infty)\rightarrow (0,\infty)$  be a some function. If $\Lambda$ has $M_\Phi$-\emph{property}, then for every $h>1$, $\varepsilon>0$, $m\in\mathbb{N}^n$ and $\delta$ with $0<\delta<c(h)$ there exist a set $E\subset \mathbb{G}^n$ and a family of sets $\{P_B:B\in\Lambda\}$ such that:
\begin{itemize}
\item[$1)$] $\delta/4^n\leq|E|\leq\delta$,

\item[$2)$] $P_B\subset\{M_B^{(\varepsilon)}(h\chi_{{}_E})>1\}$ $(B\in\Lambda)$,

\item[$3)$] $\{P_B:\;B\in\Lambda\}\subset H_j^{*}$ for some $j\in\mathbb{N}^n$ with $j\geq m$,

\item[$4)$] $|P_B|\geq c\Phi(h)|E|$ $(B\in\Lambda)$,

\item[$5)$] $|P_B\cap Q|=|P_B|\,|Q|$ $(Q\in W_m^{*}$, $B\in\Lambda)$.
\end{itemize}
\end{lemma}

\begin{proof}
Let us choose $\eta>0$ so that
\begin{equation}\label{eq:1}
    \eta<\varepsilon \;\;\text{and}\;\; 4\Big(\frac{1}{\delta c\,\Phi(h)}\Big)^{1/n}<\frac{1}{2^{m_1+\cdots+m_n}}\,.
\end{equation}
Due to the definition of $M_\Phi$-property there exist a set $E'\subset \mathbb{R}^n$ with positive measure, sets $P_B'$ $(B\in\Lambda)$ and an interval $I$ such that:
\begin{gather}
    P_B'\subset \big\{M_B^{(\varepsilon)}(h\chi_{{}_{E'}})>1\big\} \;\; (B\in\Lambda), \label{eq:2} \\
    \big\{P_B':\;B\in\Lambda\big\}\subset H_j^{*} \;\;\text{for some} \;\; j\in\mathbb{N}^n, \label{eq:3} \\
    |P_B'|\geq c\Phi(h)|E'| \;\; (B\in\Lambda), \label{eq:4} \\
    E'\subset I \;\;\text{and}\;\; P_B'\subset I \;\; (B\in\Lambda), \label{eq:5} \\
    \diam I<\eta, \label{eq:6} \\
    |E'|\geq c(h)|I|. \label{eq:7}
\end{gather}

Let $\widetilde{I}$ be the interval concentric with $I$ and such that $|E'|=\delta|\widetilde{I}|$. \eqref{eq:7} implies that $\widetilde{I}\supset I$. Put $t=\diam\widetilde{I}/\diam I$. Then by virtue of \eqref{eq:4} and \eqref{eq:5}
$$  \frac{|E'|}{\delta}=|\widetilde{I}|=t^n|I|\geq t^nc\,\Phi(h)|E'|.       $$
Therefore,
\begin{equation}\label{eq:8}
    t\leq\Big(\frac{1}{\delta c\,\Phi(h)}\Big)^{1/n}.
\end{equation}
Due to translation invariance of bases $B\in\Lambda$ we can assume that $\widetilde{I}$ has the form $(0,a_1)\times\cdots\times(0,a_n)$. By $I'$ denote the smallest among diadic intervals (i.e. among intervals from the family $\bigcup\limits_{i\in\mathbb{Z}^n} W_i)$ containing $\widetilde{I}$. Clearly, $I'\subset 4\widetilde{I}$. Therefore
\begin{equation}\label{eq:9}
    \frac{\delta}{4^n}\,|I'|\leq |E'|\leq\delta|I'|
\end{equation}
and (see \eqref{eq:8}, \eqref{eq:6} and \eqref{eq:1})
\begin{equation}\label{eq:10}
    \diam I'\leq 4\diam\widetilde{I}\leq 4\Big(\frac{1}{\delta c\,\Phi(h)}\Big)^{1/n}\diam I<\frac{1}{2^{m_1+\cdots+m_n}}\,.
\end{equation}
Let $i\in\mathbb{Z}^n$ be a $n$-tuple for which  $I'\in W_i$. \eqref{eq:10} implies that
\begin{equation}\label{eq:11}
    i>m.
\end{equation}
For each $Q\in W_i^{*}$, $T_Q$ be the translation mapping $I'$ into $Q$ and put
\begin{align*}
    E_Q & =T_Q(E') \;\; (Q\in W_i^{*}), \\
    P_{B,Q} & =T_Q(P_B') \;\; (B\in\Lambda, \;\; Q\in W_i^{*}), \\
    E & =\bigcup_{Q\in W_i^{*}} E_Q, \\
    P_B & =\bigcup_{Q\in W_i^{*}} P_{B,Q} \;\ (B\in\Lambda).
\end{align*}
Obviously,
\begin{equation}\label{eq:12}
    E_Q\subset Q \;\;\text{and}\;\; P_{B,Q}\subset Q \;\; (B\in\Lambda, Q\in W_i^{*}).
\end{equation}
By virtue of \eqref{eq:1}, \eqref{eq:2} and translation invariance of bases $B\in\Lambda$ we have
\begin{equation}\label{eq:13}
    P_{B,Q}\subset \big\{M_B^{(\eta)}(h\chi_{{}_{E_Q}})>1\big\}\subset\big\{M_B^{(\varepsilon)}(h\chi_{{}_{E_Q}})>1\big\} \;\;
            (B\in\Lambda, \;\; Q\in W_i^{*}).
\end{equation}
Since the intervals from $W_i^{*}$ are disjoint, then by \eqref{eq:3}, \eqref{eq:4}, \eqref{eq:9} and \eqref{eq:11}--\eqref{eq:13} it is easy to conclude that the sets $E$ and $P_B$ $(B\in\Lambda)$ satisfy all needed conditions.
\end{proof}

\begin{lemma}\label{lem:2}
Let $f$ be an increasing function on $[a,b]$ and $\varepsilon>0$. Then there exist points $h_1=a<h_2<\cdots<h_k=b$ such that
\begin{equation}\label{eq:2-1}
    f(h_{j+1}-)-f(h_j+)\leq\varepsilon \;\; (j=1,\dots,k-1).
\end{equation}
\end{lemma}

\begin{proof}
Let $h_1<\cdots<h_j$ are chosen. If $h_j=b$, then the construction is completed. If $h_j<b$, then let us take
$$  h_{j+1}=\sup\Big\{h\in (h_j,b]:\;f(h)-f(h_j+)\leq\varepsilon\Big\}.       $$
After same steps the construction will be completed (in opposite case we will have that $f(b)-f(a)=\infty)$. Clearly, the chosen numbers $h_1,\dots,h_k$ satisfy the condition \eqref{eq:2-1}.
\end{proof}

\begin{lemma}\label{lem:3}
Let $\Phi: (0,\infty)\rightarrow (0,\infty)$  be an increasing function, $f\in L(\mathbb{G}^n)$, $k\in\mathbb{N}$, $f/k\not\in\Phi(L)(\mathbb{G}^n)$ and $\alpha(h)>0$ $(h>1)$. Then there exist sets $A_j$ $(j\in\overline{1,m})$ and numbers $h_j$ $(j\in\overline{1,m})$ such that:
\begin{itemize}
\item[$1)$] $A_j\cap A_i=\varnothing$ $(i\neq j)$;

\item[$2)$] $k<h_j\leq|f(x)|$ $(j\in\overline{1,m}$, $x\in A_j)$;

\item[$3)$] $0<|A_j|\leq\alpha(\frac{h_j}{k})$ $(j\in\overline{1,m})$;

\item[$4)$] $\sum\limits_{j=1}^m \Phi(\frac{h_j}{k})|A_j|>k$.
\end{itemize}
\end{lemma}

\begin{proof}
Since $f/k\not\in\Phi(L)(\mathbb{G}^n)$, then there are numbers $a$ and $b$ such that
\begin{equation}\label{eq:3-1}
    k<a<b \;\;\text{and}\;\; \int\limits_{\{a\leq|f|<b\}} \Phi\Big(\frac{|f|}{k}\Big)\geq 4k.
\end{equation}
By virtue of Lemma \ref{lem:2} there are $\lambda_1=\frac{a}{k}<\lambda_2\cdots<\lambda_{p+1}=\frac{b}{k}$ with
\begin{equation}\label{eq:3-2}
    \Phi(\lambda_{q+1}-)-\Phi(\lambda_q+)<1 \;\; (q\in\overline{1,p}).
\end{equation}
For $q\in\overline{1,p}$ denote
$$  E_q=\big\{|f|=k\lambda_q\big\} \;\;\text{and}\;\; E_q'=\big\{k\lambda_q<|f|<k\lambda_{q+1}\big\}.       $$
Let us choose numbers $t_q$ and $\tau_q$ $(q\in\overline{1,p})$ so that
\begin{gather}
    \lambda_q<t_q<\tau_q<\lambda_{q+1}, \nonumber \\
    \int\limits_{E_q'\setminus\{kt_q<|f|<k\tau_q\}} \Phi\Big(\frac{|f|}{k}\Big)<\frac{1}{p}\,. \label{eq:3-3}
\end{gather}
Put $E_q^{*}=\{kt_q<|f|<k\tau_q\}$ $(q\in\overline{1,p})$. From \eqref{eq:3-2} we have: $\Phi(\tau_q)-\Phi(t_q)<1$ $(q\in\overline{1,p})$. Therefore for each $q\in\overline{1,p}$ we write
$$  \Phi(t_q)|E_q^{*}|>(\Phi(\tau_q)-1)|E_q^{*}|\geq \int\limits_{E_q^{*}} \Phi\Big(\frac{|f|}{k}\Big)-|E_q^{*}|.   $$
Consequently (see \eqref{eq:3-1} and \eqref{eq:3-3}),
\begin{multline*}
    \sum_{q=1}^p \Phi(\lambda_q)|E_q|+\sum_{q=1}^p \Phi(t_q)|E_q^{*}|\geq \\
    \geq \sum_{q=1}^p \Phi(\lambda_q)|E_q|+\sum_{q=1}^p \bigg(\int\limits_{E_q'} \Phi\Big(\frac{|f|}{k}\Big)-\frac{1}{p}-|E_q^{*}|\bigg)\geq \\
    \geq \int\limits_{\{a\leq|f|<b\}} \Phi\Big(\frac{|f|}{k}\Big)-2\geq 4k-2>k.
\end{multline*}

Denote
$$  N_1=\big\{q\in\overline{1,p}:\;|E_q|>0\big\} \;\;\text{and}\;\; N_2=\big\{q\in\overline{1,p}:\;|E_q^{*}|>0\big\}.   $$
For each $q\in N_1$, $\{E_{q,1},\dots,E_{q,\nu_q}\}$ be a partition of $E_q$ such that $0<|E_{q,\nu}|\leq\alpha(\lambda_q)$ $(\nu\in\overline{1,\nu_q})$ and for each $i\in N_2$, $\{E_{i,1}^{*},\dots,E_{i,\ell_i}\}$ be a partition of $E_i^{*}$ such that $0<|E_{i,\ell}^{*}|\leq\alpha(t_i)$ $(\ell\in\overline{1,\ell_i})$.

Put
$$  T=\big\{E_{q,\nu}:\;q\in N_1,\;\nu\in\overline{1,\nu_q}\big\}\cup\big\{E_{i,\ell}^{*}:\;i\in N_2,\;\ell\in\overline{1,\ell_i}\big\}.     $$
Let $m=\sum\limits_{q\in N_1} \nu_q+\sum\limits_{q\in N_2} \ell_i$ and $\sigma:\overline{1,m}\to T$ be a bijection. For $j\in\overline{1,m}$ denote $A_j=\sigma(j)$. Numbers $h_j$ define as follows: $h_j=k\lambda_q$ if $A_j=E_{q,\nu}$ for some $q$ and $\nu$, and $h_j=kt_i$ if $A_j=E_{i,\ell}^{*}$ for some $i$ and $\ell$. It is easy to see that sets $A_j$ and numbers $h_j$ satisfy the needed conditions.
\end{proof}

\begin{lemma}\label{lem:4}
Let $\Phi: (0,\infty)\rightarrow (0,\infty)$  be an increasing function, $f\in L(\mathbb{G}^n)$, $f/h\not\in\Phi(L)(\mathbb{G}^n)$ for every $h\geq 1$ and $\alpha(h)>0$ for every $h>1$. Then there exist a sequence of measurable sets $(A_k)$ and sequences of positive numbers $(h_k)$ and $(q_k)$ such that:
\begin{itemize}
\item[$1)$] $A_k\cap A_m=\varnothing$ $(k\neq m)$,

\item[$2)$] $q_k<h_k\leq|f(x)|$ $(k\in\mathbb{N}$, $x\in A_k)$,

\item[$3)$] $\lim\limits_{k\to\infty} q_k=\infty$,

\item[$4)$] $0<|A_k|\leq\alpha(\frac{h_k}{q_k})$ $(k\in\mathbb{N})$,

\item[$5)$] $\sum\limits_{k=1}^\infty \Phi(\frac{h_k}{q_k})|A_k|=\infty$.
\end{itemize}
\end{lemma}

\begin{proof}
It is easy to find sequences of numbers $(a_m)$ and $(b_m)$ such that:
\begin{gather}
    0<a_m<b_m<a_{m+1} \;\; (m\in\mathbb{N}), \label{eq:4-1} \\
    \int\limits_{\{a_m<|f|<b_m\}} \Phi\Big(\frac{|f|}{m}\Big)>m \;\; (m\in\mathbb{N}). \label{eq:4-2}
\end{gather}
Let $N_i$ $(i\in\mathbb{N})$ be disjoint infinite subsets of $\mathbb{N}$ and let
$$  E_i=\bigcup_{m\in N_i} \big\{a_m<|f|<b_m\big\} \;\; (i\in\mathbb{N}).       $$
From \eqref{eq:4-1} and \eqref{eq:4-2} it follows that the sets $E_i$ are disjoint and $f\chi_{{}_{E_i}}/h\not\in\Phi(L)(\mathbb{G}^n)$ for each $i\in\mathbb{N}$ and $h\geq 1$.

Let $i$ be an arbitrary natural number. Using Lemma \ref{lem:3} for parameters $\Phi$, $f\chi_{{}_{E_i}}$, $i$ and $\alpha$ we can find sets $A_{i,j}\subset E_i$ $(j\in\overline{1,m_i})$ and numbers $h_{i,j}$ $(j\in\overline{1,m_i})$ with the properties;
\begin{itemize}
\item[1)] $A_{i,j}\cap A_{i,j'}=\varnothing$ $(j\neq j')$,

\item[2)] $i<h_{i,j}\leq|f(x)|$ $(j\in\overline{1,m_i}$, $x\in A_{i,j})$,

\item[3)] $0<|A_{i,j}|\leq c(\frac{h_{i,j}}{i})$ $(j\in\overline{1,m_i})$,

\item[4)] $\sum\limits_{j=1}^{m_i} \Phi(\frac{h_{i,j}}{i})|A_{i,j}|>i$.
\end{itemize}

Making numeration by an index $k \in\mathbb{N}$ of the sequences
\begin{align*}
    & A_{1,1},\dots,A_{1,m_1};A_{2,1},\dots,A_{2,m_2};\ldots \\
    & h_{1,1},\dots,h_{1,m_1};h_{2,1},\dots,h_{2,m_2};\ldots\\
    & 1,\;\;\;\dots,\;\;\;1;\;\;\;\;2,\;\;\;\dots,\;\;\;\;2;\ldots
\end{align*}
we receive sequences $(A_k)$, $(h_k)$ and $(q_k)$ that will satisfy all needed conditions.
\end{proof}

\begin{lemma}\label{lem:5}
Let $j\in\mathbb{N}^n$, $A_1\in H_j^{*}$, $A_2\subset \mathbb{G}^n$ and $|A_2\cap Q|=|A_2|\,|Q|$ for each $Q\in W_j^{*}$. Then $|A_1\cap A_2|=|A_1|\,|A_2|$.
\end{lemma}

\begin{proof}
Let $T$ be the subfamily of $W_j^{*}$ for which $A_1=\bigcup\limits_{Q\in T} Q$. Then taking into account the condition of the lemma we have
\[  |A_1\cap A_2|=\sum_{Q\in T} |Q\cap A_2|=\sum_{Q\in T} |Q|\,|A_2|=|A_1|\,|A_2|.      \qedhere        \]
\end{proof}

\begin{lemma}\label{lem:6}
Suppose for every $k\in\mathbb{N}$ there are valid conditions: $m_k$, $j_k\in\mathbb{N}^n$, $m_k\leq j_k\leq m_{k+1}$, $A_k\in H_{j_k}^{*}$ and $|A_k\cap Q|=|A_k|\,|Q|$ for each $Q\in W_{m_k}^{*}$. Then $(A_k)$ is a sequence of independent sets.
\end{lemma}

\begin{proof}
Let $q\geq 2$ and $k_1<k_2<\cdots<k_q$. We must prove the equality
\begin{equation}\label{eq:5-1}
    \Big|\bigcap_{\nu=1}^q A_{k_\nu}\Big|=\prod_{\nu=1}^q |A_{k_\nu}|.
\end{equation}
For the case $q=2$, \eqref{eq:5-1} directly follows from Lemma \ref{lem:5}. Let us argue passing from $q-1$ to $q$. Assume that \eqref{eq:5-1} is valid for sets $A_{k_1},\dots,A_{k_{q-1}}$. It is easy to see that
$$  \bigcap_{\nu=1}^{q-1} A_{k_\nu}\in H_j^{*}, \;\;\text{where}\;\; j=j_{k_{q-1}}.     $$
Now taking into account that $m_{k_q}\geq m_{k_{q-1}+1}\geq j_{k_{q-1}}$ and $|A_{k_q}\cap Q|=|A_{k_q}|\,|Q|$ for each $Q\in W_m^{*}$, where $m=m_{k_q}$, by virtue of Lemma \ref{lem:5} and induction assumption we write
\[  \Big|\bigcap_{\nu=1}^{q} A_{k_\nu}\Big|=\Big|\bigcap_{\nu=1}^{q-1} A_{k_\nu}\Big|\,|A_{k_\nu}|=\prod_{\nu=1}^{q} |A_{k_\nu}|.   \qedhere    \]
\end{proof}

We will need the following well-known result from measure theory (see, e.g., \cite[Ch. ``Uniform Approximation'']{6} or \cite[\S~2]{7}).

\vskip+0.2cm
\noindent \textbf{Theorem A.}
\textit{For every measurable sets $A_1,A_2\subset \mathbb{R}^n$ with $|A_1|=|A_2|>0$ there exists a measure preserving and invertible mapping $\omega:A_1\to A_2$.}
\vskip+0.2cm

\section{Proof of Theorem 1}
\label{sec:4}

By Lemma \ref{lem:4} there are sequences of sets $(A_k)$ and of positive numbers $(h_k)$ and $(q_k)$ such that:
\begin{gather}
    A_k\cap A_m=\varnothing \;\; (k\neq m), \label{eq:th-1} \\
    q_k<h_k\leq|f(x)| \;\; (k\in\mathbb{N}, \;\; x\in A_k), \label{eq:th-2} \\
    \lim_{k\to\infty} q_k=\infty, \label{eq:th-3} \\
    0<|A_k|\leq c\Big(\frac{h_k}{q_k}\Big) \;\; (k\in\mathbb{N}), \nonumber \\
   \sum_{k=1}^\infty \Phi\Big(\frac{h_k}{q_k}\Big)|A_k|=\infty. \label{eq:th-4}
\end{gather}
According to Lemma \ref{lem:1}, for every $k\in\mathbb{N}$ and $m_k\in\mathbb{N}^n$ there exist a set $E_k$ and a family of sets $\{P_{B,k}:\;B\in\Lambda\}$ with the properties:
\begin{gather}
    \big\{P_{B,k}:\;B\in\Lambda\big\}\subset H_{j_k}^{*} \;\;\text{for some}\;\;
            j_k\in\mathbb{N}^n \;\;\text{with}\;\; j_k\geq m_k, \label{eq:th-5} \\
    \frac{|A_k|}{4^n}\leq |E_k|\leq |A_k|, \label{eq:th-6} \\
    P_{B,k}\subset\bigg\{M_B^{(1/k)}\Big(\frac{h_k}{q_k}\,\chi_{{}_{E_k}}\Big)>1\bigg\}=
            \big\{M_B^{(1/k)}(h_k\chi_{{}_{E_k}})>q_k\big\} \;\; (B\in\Lambda), \label{eq:th-7} \\
    |P_{B,k}|\geq c\Phi\Big(\frac{h_k}{q_k}\Big)|A_k| \;\; (B\in\Lambda), \label{eq:th-8} \\
    |P_{B,k}\cap Q|=|P_{B,k}|\,|Q| \;\; (B\in\Lambda, \;\; Q\in W_{m_k}^{*}). \label{eq:th-9}
\end{gather}
From \eqref{eq:th-4}, \eqref{eq:th-6} and \eqref{eq:th-8},
\begin{equation}\label{eq:th-10}
    \sum_{k=1}^\infty |P_{B,k}|=\infty \;\; (B\in\Lambda).
\end{equation}
Obviously, we can choose $(m_k)$ so that $m_{k+1}\geq j_k$ $(k\in\mathbb{N})$. Then by \eqref{eq:th-5}, \eqref{eq:th-9} and Lemma \ref{lem:6}, $(P_{B,k})$ is a sequence of independent sets for every $B\in\Lambda$. Therefore by virtue of \eqref{eq:th-10} and Borel--Cantelli lemma we have
\begin{equation}\label{eq:th-11}
    \Big|\olim_{k\to\infty} P_{B,k}\Big|=1 \;\;\text{for every}\;\; B\in\Lambda.
\end{equation}

Put $g=\sup\limits_{k\in\mathbb{N}} h_k\chi_{{}_{E_k}}$. Since $g\leq\sum\limits_{k=1}^\infty h_k\chi_{{}_{E_k}}$, then by  \eqref{eq:th-1} and \eqref{eq:th-6},
$$  \int\limits_{\mathbb{G}^n} g\leq \sum_{k=1}^\infty h_k|E_k|\leq \sum_{k=1}^\infty h_k|A_k|\leq \int\limits_{\mathbb{G}^n} |f|<\infty.   $$
Thus $g\in L(\mathbb{G}^n)$. Let $B\in\Lambda$ and $x\in\olim\limits_{k\to\infty} P_{B,k}$. Then by \eqref{eq:th-3} and \eqref{eq:th-7}, $\overline{D}_B(\int g,x)=\infty$. Consequently, taking into account \eqref{eq:th-11} we conclude that $g\in S_\Lambda$. Obviously, we have also that
\begin{equation}\label{eq:th-12}
    g\chi_{{}_{\{g<\infty\}}}\in S_\Lambda.
\end{equation}

Denote
$$  E=\bigcup_{k=1}^\infty E_k \;\;\text{and}\;\; E_k'=E_k\setminus \bigcup_{j>k} E_j \;\; (k\in\mathbb{N}).        $$
It is easy to check that
\begin{gather*}
     E_k'\cap E_m'=\varnothing \;\; (k\neq m), \\
     g\chi_{{}_{\{g<\infty\}}}=\sum_{k=1}^\infty h_k\chi_{{}_{E_k'}}.
\end{gather*}
For each $k\in\mathbb{N}$ let us choose a measurable set $A_k'$ so that $A_k'\subset A_k$ and $|A_k'|=|E_k'|$. Denote $A=\bigcup\limits_{k=1}^\infty A_k'$. Due to Theorem A there exist a measure preserving and invertible mappings $\omega_k:A_k'\to E_k'$ $(k\in\mathbb{N})$ and $\omega_0:\mathbb{G}^n\setminus A\to\mathbb{G}^n\setminus E$. A mapping $\omega$ define as follows
$$  \omega(x)=\begin{cases}
        \omega_k(x) & (k\in\mathbb{N}, \;\; x\in A_k'), \\
        \omega_0(x) & (x\in\mathbb{G}^n\setminus A), \\
        x & (x \in \mathbb{R}^n\setminus \mathbb{G}^n).
            \end{cases}     $$
It is easy to check that: 1) $\omega$ is measure preserving and invertible; and 2)$|f|\circ\omega\geq g$. Now taking into account \eqref{eq:th-12} we conclude the validity of the theorem.\ \hfill \qed

\begin{remark}\label{rem:1-1}
For the case of finite or countable family $\Lambda$,  Theorem 1 can be strengthened by achieving divergence at every point with respect  to each $B\in\Lambda$, i.e. a mapping $\omega$  can be chosen so that for every basis $B\in\Lambda$ the equality $\overline{D}_B\left(\int |f|\circ \omega,x\right)=\infty$ would be fulfilled at every $x\in\mathbb{G}^n$.
\end{remark}

\section{Applications of Theorem \ref{th:1}}
\label{sec:5}

\noindent \textbf{5.1.} Zygmund problem in general setting may be formulated as follows: \emph{Let $B$ be a translation invariant basis in $\mathbb{R}^n$ and let  $\Delta(B)=\big\{B(\gamma):\;\gamma\in\Gamma_n\big\}$. Is the class $S_{\Delta(B)}$ non-empty? }

Below it is found  a quite general condition for a basis $B$(see Corollary 1) fulfillment of which provide the positive answer to the posed question.

For a basis $B$ denote
$$  \Phi_B(h)=\lim_{t\to\infty} \olim_{r\to 0} \frac{|\{M_B^{(tr)}(h\chi_{{}_{V_r}})>1\}|}{|V_r|} \;\; (h>0),      $$
where $V_r=\{x\in\mathbb{R}^n:\dist(x,0)<r\}$. Note that:
\begin{enumerate}
\item[1)] $\Phi_B$ is increasing;

\item[2)] If $B$ is a convex basis, then by virtue of the estimation $M_B(\chi_{{}_{V_r}})(x)< cr/\dist(x,V_r)$ $(x\not\in V_{2r})$ (see \cite[Lemma 1]{8}) we have
$$  \Phi_B(h)=\olim_{r\to 0} \frac{|\{M_B^{(tr)}(h\chi_{{}_{V_r}})>1\}|}{|V_r|};        $$

\item[3)] If $B$ is translation and homothecy invariant, then
$$  \Phi_B(h)=\big|\big\{M_B^{(tr)}(h\chi_{{}_{V}})>1\big\}\big|,     $$
where $V=\{x\in\mathbb{R}^n:\dist(x,0)<1\}$.
\end{enumerate}

Let us call sets from a class $\Delta$ \emph{uniformly measurable in Jordan sense} if for every $\varepsilon>0$ there exist $k\in\mathbb{N}$ and $\varepsilon>0$ such that:
\begin{enumerate}
\item[1)] $kr^n<\varepsilon$;

\item[2)] for every $E\in\Delta$, there exists a cover of $\partial E$ consisting of $k$ balls with radius $\varepsilon$.
\end{enumerate}

\begin{remark}\label{rem:1}
If $\Delta$ is a collection of measurable in Jordan sense and mutually congruent sets, then it is easy to see that the sets from $\Delta$ are uniformly measurable in Jordan sense.
\end{remark}

\begin{lemma}\label{lem:7}
Let $\Delta$ be a non-empty family of sets that are uniformly measurable in Jordan sense and let $\inf\limits_{E \in\Delta}|E|>0$. Then for every $\varepsilon>0$ there exists $m_0\in\mathbb{N}^n$ such that
$$  \Big|\bigcup_{Q\in W_m,\,Q\subset E} Q\Big|>(1-\varepsilon)|E| \;\;\text{and}\;\;
            \Big|\bigcup_{Q\in W_m,\,Q\cap E\neq\varnothing} Q\Big|<(1+\varepsilon)|E|      $$
for every $E\in\Delta$ and $m\geq m_0$.
\end{lemma}

\begin{proof}
Denote $t=\inf\limits_{E\in\Delta}|E|$ and $V(x,r)=\{y\in\mathbb{R}^n:\dist(y,x)<r\}$. By virtue of the lemma condition there are $k\in\mathbb{N}$ and $r>0$ such that
\begin{equation}\label{eq:7-1}
    4^nkr^n<\varepsilon t,
\end{equation}
and for every $E\in\Delta$ we can choose points $x_{E,1},\dots,x_{E,k}$ for which
\begin{equation}\label{eq:7-2}
    \partial E\subset \bigcup_{j=1}^k V(x_{E,j},r).
\end{equation}

Let $m_0\in\mathbb{N}^n$ be such that $\diam Q<r$ if $Q\in W_{m_0}$. Suppose $m\geq m_0$. For $E\in\Delta$ denote
$$  A_E=\Big\{Q\in W_m:\;Q\cap\bigcup_{j=1}^k V(x_{E,j},r)\neq\varnothing\Big\}.       $$
By choosing of $m_0$ we have
$$  \bigcup_{Q \in A_E} Q\subset \bigcup_{j=1}^k V(x_{E,j},2r).      $$
Consequently, by \eqref{eq:7-1}
\begin{equation}\label{eq:7-3}
    \Big|\bigcup_{Q \in A_E} Q\Big|\leq 2^n\sum_{j=1}^k \big|V(x_{E,j},r)\big|\leq 2^nk2^nr^n<\varepsilon t.
\end{equation}
Note that if $Q\in W_m$, $Q\cap E\neq\varnothing$ and $Q\cap(\mathbb{R}^n\setminus E)\neq\varnothing$, then $Q\cap\partial E\neq\varnothing$. Therefore, by \eqref{eq:7-2},
$$  \big\{Q\in W_m:\;Q\subset E\big\}\supset \big\{Q\in W_m:\;Q\cap E\neq\varnothing\big\}\setminus A_E   $$
and
$$  \big\{Q\in W_m:\;Q\cap E\neq\varnothing\big\}\subset \big\{Q\in W_m:\;Q\subset E\big\}\cup A_E.     $$
Consequently, taking into account \eqref{eq:7-3}, we obtain
$$  \Big|\bigcup_{Q\in W_m,\,Q\subset E} Q\Big|\geq \Big|\bigcup_{Q\in W_m,\,Q\cap E\neq\varnothing} Q\Big|-\Big|\bigcup_{Q\in A_E} Q\Big|>
            |E|-\varepsilon t\geq (1-\varepsilon)|E|        $$
and
\[  \Big|\bigcup_{Q\in W_m,\,Q\cap E\neq\varnothing} Q\Big|\leq \Big|\bigcup_{Q\in W_m,\,Q\subset E} Q\Big|+\Big|\bigcup_{Q\in A_E} Q\Big|<
            |E|+\varepsilon t\leq (1+\varepsilon)|E|.        \qedhere       \]
\end{proof}

\begin{lemma}\label{lem:8}
Let $B$ be a translation invariant basis in $\mathbb{R}^n$. Then the family $\Delta(B)$ has $M_{\Phi_B}$-property.
\end{lemma}

\begin{proof}
Let $h>1$ and $\varepsilon>0$. Take $t>1$ such that
$$  \olim_{r\to 0} \frac{|\{M_B^{(tr)}(h\chi_{{}_{V_r}})>1\}|}{|V_r|}>\frac{\Phi_B(h)}{2}\,.     $$

Let us consider $r>0$ for which
\begin{equation}\label{eq:8-1}
    2\sqrt{n}\,r(1+t)<\varepsilon \;\;\text{and}\;\; \big|\big\{M_B^{(tr)}(h\chi_{{}_{V_r}})>1\big\}\big|>\frac{\Phi_B(h)}{2}\,|V_r|.
\end{equation}

It is easy to check that for every $f\in L(\mathbb{R}^n)$, $\delta>0$, $x\in\mathbb{R}^n$ and $y\in \Gamma_n$,
$$  M_B^{(\delta)}(f)(x)=M_{B(\gamma)}^{(\delta)}(f\circ\gamma^{-1})(\gamma(x)).        $$
Therefore, we get
$$  M_B^{(tr)}(h\chi_{{}_{V_r}})(x)=M_{B(\gamma)}^{(tr)}(h\chi_{{}_{V_r}})(\gamma(x)) \;\; (x\in\mathbb{R}^n, \;\; \gamma\in\Gamma_n).      $$
Consequently,
\begin{equation}\label{eq:8-2}
    \big\{M_{B(\gamma)}^{(tr)}(h\chi_{{}_{V_r}})>1\big\}=\gamma\Big(\big\{M_B^{(tr)}(h\chi_{{}_{V_r}})>1\big\}\Big) \;\; (\gamma\in\Gamma_n).
\end{equation}

Since the set $\{M_B^{(tr)}(h\chi_{{}_{V_r}})>1\}$ is open, then there exists a set $A$ that is a finite union of cubic intervals such that
\begin{equation}\label{eq:8-3}
    A\subset \big\{M_B^{(tr)}(h\chi_{{}_{V_r}})>1\big\} \;\;\text{and}\;\; |A|>\frac{\Phi_B(h)}{2}\,|V_r|.
\end{equation}
Put $A_\gamma=\gamma(A)$ $(\gamma\in\Gamma_n)$. Since sets $A_\gamma$ are measurable in Jordan sense and mutually congruent, then by Lemma \ref{lem:7} we conclude the existence of $m\in\mathbb{N}^n$ and sets $P_\gamma$ $(\gamma\in\Gamma_n)$ such that
\begin{gather}
    P_\gamma\subset A_\gamma \;\;\text{and}\;\; |P_\gamma|>\frac{|A_\gamma|}{2} \;\;(\gamma\in\Gamma_n), \label{eq:8-4} \\
    \big\{P_\gamma:\;\gamma\in\Gamma_n\big\}\subset H_m. \label{eq:8-5}
\end{gather}
From \eqref{eq:8-1}--\eqref{eq:8-4} we have that for every $\gamma\in\Gamma_n$,
\begin{gather}
    P_\gamma\subset \big\{M_{B(\gamma)}^{(\varepsilon)}(h\chi_{{}_{V_r}})>1\big\}, \label{eq:8-6} \\
    |P_\gamma|>\frac{\Phi_B(h)}{4}\,|V_r|, \label{eq:8-7} \\
    V_r\cup\bigcup_{\gamma\in\Gamma_n} P_\gamma\subset V_{r(1+t)}. \label{eq:8-8}
\end{gather}
Assuming $c=1/4$, $c(h)=\frac{1}{2^n n^{n/2}(1+t)^n}$\,, $E=V_r$, $P_{B(\gamma)}=P_\gamma$ $(\gamma\in\Gamma_n)$ and $Q=(-r(1+t),r(1+t))^n$, from \eqref{eq:8-5}--\eqref{eq:8-8} and \eqref{eq:8-1} we conclude that the family $\Lambda=\{B(\gamma):\;\gamma\in\Gamma_n\}$ has $M_{\Phi_B}$-property.
\end{proof}

From Theorem \ref{th:1} on the basis of Lemma \ref{lem:8} we obtain the following result.

\begin{theorem}\label{th:2}
Let $B$ be a translation invariant basis in $\mathbb{R}^n$. If the function $\Phi_B$ is non-regular, then for every $f\in L\setminus[\Phi_B(L)](\mathbb{G}^n)$ there exists a measure preserving and invertible mapping $\omega:\mathbb{R}^n\to\mathbb{R}^n$ such that  $\{x:\omega(x)\neq x\}\subset \mathbb{G}^n$ and $|f|\circ\omega\in S_{\Delta(B)}$. In particular, if \;$\Phi_B$ additionally satisfies $\Delta_2$-condition at infinity, then the same conclusion is valid for every $f\in L\setminus \Phi_B(L)(\mathbb{G}^n)$.
\end{theorem}

\begin{corollary}\label{cor:1}
Let $B$ be a translation invariant basis in $\mathbb{R}^n$. If the function $\Phi_B$ is non-regular, then the class $S_{\Delta(B)}$ is non-empty.
\end{corollary}

\medskip
\medskip

\noindent \textbf{5.2.} It is true the following theorem.

\begin{theorem}\label{th:3}
Let $B$ be a translation invariant basis in $\mathbb{R}^n$. If the function $\Phi_B$ is non-regular, then for every Orlicz space $\psi(L)(\mathbb{G}^n)$ with the properties: $\psi$ satisfies $\Delta_2$-condition at infinity and $\lim\limits_{h\to\infty} \frac{\psi(h)}{\Phi_B(h)}=0$, the set $\psi(L)(\mathbb{G}^n)\setminus S_{\Delta(B)}$ is of the first category in $\psi(L)(\mathbb{G}^n)$.
\end{theorem}

\begin{remark}\label{rem:2}
Theorem \ref{th:3} generalizes the result of $B$. L\'{o}pes Melero \cite{9} which asserts the same for the case when the following weak variant of the function $\Phi_B$ is considered
$$  \widetilde{\Phi}_B(h)=\olim_{r\to 0} \frac{|\{M_B^{(hr)}(h\chi_{{}_{V_r}})>1\}|}{|V_r|}\,.      $$
\end{remark}

\begin{proof}[Proof of Theorem $\ref{th:3}$]
For $k\in\mathbb{N}$ by $E_k$ denote the set of all functions $f\in \psi(L)(\mathbb{G}^n)$ for which there is a set $A=A(f)\subset \mathbb{G}^n$ and a rotation $\gamma=\gamma(f)\in\Gamma_n$ with the properties:
\begin{itemize}
\item[1)] $|A|\geq\frac{1}{k}$\,;

\item[2)] is $x\in A$, $R\in B(\gamma)(x)$ and $\diam R\leq\frac{1}{k}$, then $\frac{1}{|R|}\int\limits_R f\leq k$.
\end{itemize}

It is easy to see that $\Psi(L)\setminus S_{\Delta(B)}=\bigcup\limits_{k=1}^\infty E_k$. Therefore it sufficies to prove that $E_k$ is nowhere dense in $\psi(L)$ for every $k\in\mathbb{N}$.

Let $k\in\mathbb{N}$. First let us prove the closeness of $E_k$. Suppose $f_j\in E_k$ $(j\in\mathbb{N})$, $f\in\psi(L)$ and $\|f_j-f\|_{\psi(L)}\to 0$. Clearly, we can choose subsequence of $(\gamma(f_j))$ which is convergent by the natural metric in $\Gamma_n$. Without loss of generality assume that $\gamma(f_j)$ converges and let $\gamma$ be its limit. By $A$ denote the set $\olim\limits_{j\rightarrow\infty} A(f_j)$. Obviously, $A\subset \mathbb{G}^n$ and $|A|\geq\frac{1}{k}$\,. Take $x\in A$ and $R\in B(\gamma)(x)$ with $\diam R<\frac{1}{k}$\,. Without loss of generality assume that $x\in A(f_j)$ for every $j\in\mathbb{N}$. Let us consider the set $T\in B(0)$ for which $R=x+\gamma(T)$. Put $R_j=x+\gamma(f_j)(T)$ $(j\in\mathbb{N})$. Then $R_j\in B(\gamma(f_j))(x)$ and $\diam R_j\leq\frac{1}{k}$\,. Consequently,
$$  \frac{1}{|R_j|}\int\limits_{R_j} f\leq k.       $$
Now taking into account that $|(R_j\setminus R)\cup(R\setminus R_j)|\to 0$ and $\|f_j-f\|_L\leq c\|f_j-f\|_{\psi(L)}\rightarrow 0$, we obtain
$$  \frac{1}{|R|}\int\limits_{R} f=\lim_{j\to\infty} \frac{1}{|R_j|}\int\limits_{R_j} f_j\leq k.       $$

Thus the closeness of $E_k$ is proved. The next step is to prove that $\psi(L)\setminus E_k$ is dense in $\psi(L)$. Take a function $f\in\psi(L)$ and $\varepsilon>0$. Since $\psi$ satisfies $\Delta_2$-condition at infinity, then (see e.g. \cite[\S~4]{10}) there is $g\in L^\infty$ with $\|f-g\|_{\psi(L)}<\varepsilon/2$\,. Let us consider a function $\ell\in S_{\Delta(B)}$ with $\|\ell\|_{\psi(L)}<\varepsilon/2$\,. Existence of such function is provided by Theorem \ref{th:2}. Obviously, $g+\ell\in S_{\Delta(B)}$. Now from the estimation $\|f-(g+\ell)\|_{\psi(L)}<\varepsilon$ we conclude the density of $S_{\Delta(B)}$ in $\psi(L)$. Consequently, using inclusion $S_{\Delta(B)}\subset \psi(L)\setminus E_k$ we obtain the density of $\psi(L)\setminus E_k$ in $\psi(L)$. Finally, taking into account the closeness of $E_k$ we conclude that $E_k$ is nowhere dense in $\psi(L)$.
\end{proof}

\noindent \textbf{5.3.} In this section we will apply Theorems 2 and 3 for bases $\bold{I}_n^k$.

\begin{lemma}\label{lem:9}
Let $\delta_1,\dots,\delta_n>0$, $h>1$, and let $E$ be the set of all points $x\in\mathbb{R}^n$ such that
\begin{gather*}
    x_1>\delta_1,\dots,x_n>\delta_n, \\
    x_1\cdots x_n<h\delta_1\cdots\delta_n.
\end{gather*}
Then
$$  \int\limits_E \frac{1}{x_1\cdots x_n}\,dx_1\cdots dx_n>c(\ln h)^n,     $$
where $c$ is a positive number depending only on $n$.
\end{lemma}

\begin{proof}
For $n=1$ the assertion is obvious. Let us consider the passing from $n-1$ to $n$.

For $\delta_n<t<h\delta_n$ denote
$$  A_t=\Big\{x\in\mathbb{R}^{n-1}:\;x_1>\delta_1,\dots,x_{n-1}>\delta_{n-1},\;x_1\cdots x_{n-1}<\frac{h\delta_n}{t}\,\delta_1\cdots\delta_{n-1}\Big\}.   $$
By Fubini theorem and the induction assumption we have
\begin{multline*}
    \int\limits_E \frac{1}{x_1\cdots x_n}\,dx_1\cdots dx_n= \\
    =\int\limits_1^{h\delta_n} \frac{1}{x_n}\,\bigg[\int\limits_{A_t} \frac{1}{x_1\cdots x_{n-1}}\,dx_1\cdots dx_{n-1}\bigg]\,dx_n> \\
    >\int\limits_{\delta_n}^{h\delta_n} c_{n-1}\,\frac{1}{t}\,\Big(\ln\frac{h\delta_n}{t}\Big)^{n-1}\,dt>
            c_{n-1}\int\limits_{\delta_n}^{\sqrt{h}\,\delta_n} \frac{1}{t}\,(\ln\sqrt{h})^{n-1}\,dt\geq c_n(\ln h)^n. \qedhere
\end{multline*}
\end{proof}

\begin{lemma}\label{lem:10}
For every $k$, $1\leq k\leq n-1$, and an interval $I$ of type
$$  I=\underset{j=1}{\overset{n}{\times}} (a_j,a_j+\delta_j), \;\;\text{where}\;\; \delta_k=\delta_{k+1}\cdots=\delta_n,      $$
and $h>1$ it is valid the estimation
$$  \big|\big\{M_{\bold{I}_n^k}(h\chi_{{}_I})>1\big\}\big|\geq ch(\ln h)^{k}|I|,     $$
where $c$ is a positive number depending only on $n$.
\end{lemma}

\begin{proof}
Without loss of generality assume that $h>2^n$ and $I$ is of type
$$  I=(0,\delta_1)\times\cdots\times(0,\delta_k)\times(0,\delta_k)\times\cdots\times(0,\delta_{k}).       $$
Let $x\in\mathbb{R}^n$ be such that
\begin{gather}
    x_1> \delta_1,\cdots, x_k>\delta, \;\; x_{k+1}> \delta_k,\dots,x_n>\delta_{k}, \label{eq:10-1} \\
    x_{1}\dots x_k \big[\max(x_{k+1},\dots, x_n)\big]^{n-k} <h|I|. \label{eq:10-2}
\end{gather}
Put
$$  J=(0,x_{1})\times\cdots\times(0,x_k)\times\big(0,\max(x_{k+1},\dots,x_n)\big)^{n-k}. $$
By \eqref{eq:10-1} and \eqref{eq:10-2}, $J\supset I$ and $|J|<h|I|$. Consequently,
$$  \frac{1}{|J|} \int\limits_J h\chi_{{}_I}=\frac{h|I|}{|J|}>1.        $$
Thus $M_{\textbf{I}_n^k}(h\chi_{{}_I})(x)>1$, and therefore
$$  \big\{M_{\textbf{I}_n^k}(h\chi_{{}_I})(x)>1\big\}\supset\big\{x\in\mathbb{R}^n:\;x\;\text{satisfies \eqref{eq:10-1} and \eqref{eq:10-2}}\big\}\equiv E. $$

Let us estimate $|E|$. Denote
\begin{align*}
    T & =\Big\{y\in\mathbb{R}^{k}:\;y_1>\delta_1,\dots,y_{k}>\delta_{k},\;y_1\cdots y_{k}<\frac{h}{2^k}\,\delta_1\cdots\delta_{k}\Big\}, \\
    E_y & =\Big\{x\in E:\;(x_{1},\dots,x_k)=y\Big\} \;\; (y\in T).
\end{align*}
It is easy to see that for every $y\in T$,
$$  |E_y|_{n-k}=\Big(\Big(\frac{R|I|}{y_1\cdots y_{k}}\Big)^{1/(n-k)}-\delta\Big)^{n-k}>\frac{1}{2^n}\,\frac{R|I|}{y_1\cdots y_{k}}\,.  $$
Therefore, using Fubini theorem we have
\begin{align*}
    |E| & >\big|\big\{x\in E:\;(x_{1},\dots,x_k)\in T\big\}\big|= \\
    & =\int\limits_T |E_y|_{n-k}\,dy>\frac{h|I|}{2^n}\int\limits_T \frac{1}{y_1\cdots y_{k}}\,dy_1\cdots dy_{k}.
\end{align*}
Consequently, by virtue of Lemma \ref{lem:10} we conclude the validity of the needed estimation.
\end{proof}

\begin{lemma}\label{lem:11}
For every $k$ with $2\leq k\leq n$ there are valid the estimations
$$  c_1h(\ln h)^{k-1}\leq\Phi_{\bold{I}_n^k}(h)\leq c_2h(\ln h)^{k-1} \;\; (h>1),      $$
where $c_1$ and $c_2$ are positive numbers depending only on $n$.
\end{lemma}

\begin{proof}
Let $r>0,$ $h>1$ and $Q=(-\frac{r}{\sqrt{n}}\,,\frac{r}{\sqrt{n}})^n$. By Lemma \ref{lem:10} we have
\begin{multline}
    \big|\big\{M_{\textbf{I}_n^k}(h\chi_{{}_{V_r}})>1\big\}\big|\geq \big|\big\{M_{\textbf{I}_n^k}(h\chi_{{}_Q})>1\big\}\big|\geq \\
    \geq ch(\ln h)^{k-1}|Q|\geq c_1h(\ln h)^{k-1}|V_r|, \label{eq:11-1}
\end{multline}
where $c_1>0$ depends only on $n$. On the other hand by virtue of the well-known estimation (see \cite[Chapter~II, \S~3]{3})
$$  \big|\big\{M_{\textbf{I}_n^k}(f)>\lambda\big\}\big|\leq
        c\int\limits_{\mathbb{R}^n} \frac{|f|}{\lambda}\,\Big(1+\ln\frac{|f|}{\lambda}\Big)^{k-1} \;\; (f\in L(\mathbb{R}^n), \;\;\lambda>0), $$
it follows that
\begin{equation}\label{eq:11-2}
    \big|\big\{M_{\textbf{I}_n^k}(h\chi_{{}_{V_r}})>1\big\}\big|\leq c_2h(\ln h)^{k-1}|V_r|,
\end{equation}
where $c_2>0$ depends only on $n$.

From \eqref{eq:11-1} and \eqref{eq:11-2} we conclude the validity of the lemma.
\end{proof}

From Theorems \ref{th:2} and \ref{th:3} on the basis of Lemma \ref{lem:11} we obtain the following result.

\begin{theorem}\label{th:4}
Let $2\leq k\leq n$. Then:
\begin{enumerate}
\item[$1)$] for every function $f\in L\setminus L(\ln^{+}L)^{k-1}(\mathbb{G}^n)$, there exists a measure preserving and invertible mapping $\omega:\mathbb{R}^n\to\mathbb{R}^n$ such that  $\{x:\omega(x)\neq x\}\subset \mathbb{G}^n$ and $|f|\circ\omega\in S_{\Delta(\bold{I}_n^k)}$;

\item[$2)$] for every Orlicz space $\psi(L)(\mathbb{G}^n)$ with the properties: $\psi$ satisfies $\Delta_2$-con\-di\-ti\-on at infinity and $\lim\limits_{h\to\infty} \frac{\psi(h)}{h(\ln h)^{k-1}}=0$, the set $\psi(L)(\mathbb{G}^n)\setminus S_{\Delta(\bold{I}_n^k)}$ is of the first category in $\psi(L)(\mathbb{G}^n)$.
\end{enumerate}
\end{theorem}

\begin{remark}\label{rem:3}
The first part of Theorem \ref{th:1} was announced in \cite{11}.
\end{remark}

\section*{Acknowledgements}

The author was supported by Shota Rustaveli National Science Foundation (Project \#~31/48).

%\vskip+0.5cm

%\noindent \textbf{Authors' addresses:}

%\vskip+0.3cm

%Akaki Tsereteli State University, 59 Tamar Mepe St., Kutaisi 4600, Georgia.

%\textit{E-mail:} \texttt{oniani@atsu.edu.ge}

\end{document}